\documentclass[a4paper, preprint]{elsarticle}

\usepackage{amssymb,amsmath,amsthm, color}
\usepackage{url}
\usepackage[utf8]{inputenc}

\newtheorem{theorem}{Theorem}[section]
\newtheorem{lemma}[theorem]{Lemma}
\newtheorem{define}[theorem]{Definition}
\newtheorem{cor}[theorem]{Corollary}
\newtheorem{prop}[theorem]{Proposition}
\newtheorem{remark}[theorem]{Remark}

\newcommand{\I}{\mathcal I}
\newcommand{\F}{\mathbb F}

\newcommand{\fqn}{\mathbb F_{q^n}}
\newcommand{\Z}{\mathbb Z}

\newcommand{\GL}{\mathrm{GL}}
\newcommand{\Gal}{\mathrm {Gal}}
\newcommand{\PGL}{\mathrm{PGL}}
\newcommand{\PTL}{\mathrm{P\Gamma{L}}}

\newcommand{\ord}{\mathrm{ord}}

\newcommand{\doublespace}

\begin{document}

\begin{frontmatter}

\title{M\"obius-Frobenius maps on irreducible polynomials}
\author[UFMG]{F.E. Brochero Mart\'{i}nez}
\ead{fbrocher@mat.ufmg.br}
\author[UFMG]{Daniela Oliveira}
\ead{danielaalvesoliveira@gmail.com}
\author[USP]{Lucas Reis\corref{cor1}}
\ead{lucasreismat@gmail.com}
\cortext[cor1]{Corresponding author}

\address[UFMG]{Departamento de Matem\'{a}tica, Universidade Federal de Minas Gerais, Belo Horizonte, MG, 30123-970, Brazil.}
\address[USP]{Universidade de S\~{a}o Paulo, Instituto de Ci\^{e}ncias Matem\'{a}ticas e de Computa\c{c}\~{a}o, S\~{a}o
Carlos, SP 13560-970, Brazil.}
\journal{Elsevier}
\begin{abstract}
Let $n$ be a positive integer and let $\F_{q^n}$ be the finite field with $q^n$ elements, where $q$ is a prime power. This paper introduces a natural action of the \emph{Projective Semilinear Group} $\PTL(2, q^n)=\PGL(2, q^n)\rtimes \Gal(\fqn/\F_q)$ on the set of monic irreducible polynomials over the finite field $\F_{q^n}$. Our main results provide information on the characterization and number of fixed points.

\end{abstract}

\begin{keyword}
 irreducible polynomials over finite fields; group action; fixed points; enumeration formula
\MSC[2010]{11T06 \sep 11T55\sep 12E20}
\end{keyword}
\end{frontmatter}





\section{Introduction}
Let $\F_Q$ be the finite field of $Q$ elements, where $Q$ is a prime power and, for each positive integer $k\ge 1$, let $\I(Q, k)$ be denote the set of monic irreducible polynomials of degree $k$ over $\F_Q$. As pointed out in \cite{ST12}, the Projective Linear Group $\PGL(2, Q)$ \emph{acts} on the sets $\I(Q, k), k\ge 2$ via the following M\"obius-like maps: for $[A]\in \PGL(2, Q)$ with $A=\left(\begin{matrix}
a&b\\
c&d
\end{matrix}\right)$  and $f\in \I(Q, k)$, we define $[A]\circ f$ as the unique monic polynomial that is a scalar multiple of the polynomial
$$A\circ f(x):= (cx+d)^k\cdot f\left(\frac{ax+b}{cx+d}\right).$$
In the past few years, many authors studied this action (and some other similar), focusing on problems regarding the characterization and number of the $[A]$-invariants, i.e., monic irreducible polynomials $f$ for which $[A]\circ f=f$. For more details, see \cite{Gar11, MP17, LR17, ST12}. 

We observe that there is a natural extension of this action. Namely, fix $q$ a prime power and $n$ a positive integer. If $\Gal(\fqn/\F_q)$ denotes the Galois Group of the extension $\fqn/\F_q$, such a group is cyclic, has order $n$ and is generated by the Frobenius automorphism $\sigma_1:\F_{q^n}\to \fqn$ with $\alpha\mapsto \alpha^q$. For each $i\ge 1$, let $\sigma_i$ be denote the $i$-th composition of $\sigma_1$. We observe that $\sigma_i$ naturally extends to the polynomial ring $\fqn[x]$ and, for simplicity, $\sigma_i$ also denotes this extension. In particular, we observe that the {\em Projective Semilinear Group} $\PTL(2, q^n)=\PGL(2, q^n)\rtimes \Gal(\fqn/\F_q)$ induces maps of M\"obius-Frobenius type on the set of monic irreducible polynomials over $\F_{q^n}$. Namely, for $f\in \I_k:=\I(q^n, k)$ with $k\ge 2$ and $\mathfrak g\in \PTL(2, q^n)$ with $\mathfrak g=[A, \sigma_i]$, we define the following composition
$$[A, \sigma_i]\ast f(x)=[A]\circ (\sigma_i(f)).$$

We shall prove that these compositions yield an action of the group $\PTL(2, q^n)$ on the sets $\I_k$ with $k\ge 2$ and, in this paper, we are interested in the characterization and number of fixed points. For $k\ge 2$ and $[A, \sigma_i]\in \PTL(2, q^n)$, we say that $f\in \I_k$ is $[A, \sigma_i]$-invariant if $[A, \sigma_i]\ast f=f$. We further prove that it suffices to consider the case $i=1$, i.e., we only need to explore $[A, \sigma_1]$-invariants for generic $[A]\in \PGL(2, q^n)$. Our main result provide an asymptotic formula for the number of invariants of an arbitrary degree and can be stated as follows.

\begin{theorem}\label{thm:main}
Let $[A, \sigma_1]\in \PTL(2, q^n)$, set $A^*=A\sigma_1(A)\ldots \sigma_{n-1}(A)$ and $D=\ord([A^*])$. Then, for any $k>2$, the number $n_{A}(k)$ of $[A, \sigma_1]$-invariants of degree $k$ equals zero if $k$ is not of the form $Ds$ with $\gcd(s, n)=1$. If $k>2$ is of the form $Ds$ with $\gcd(s, n)=1$, then
$$n_{A}(Ds)\approx \frac{\varphi(D)}{Ds}q^s,$$
where $a_s\approx b_s$ means that $\lim\limits_{s\to \infty}\frac{a_s}{b_s}=1$.
\end{theorem}
We observe that the case $n=1$ is just the action of the group $\PGL(2, q)$ that was previously studied in~\cite{ST12}. In this context, Theorem~\ref{thm:main} generalizes Theorem~5.3 of~\cite{ST12}. We further explore the $\mathfrak g$-invariants, where $\mathfrak g$ is an element of the form $[A, \sigma_i]$ with $[A]\in \PGL(2, q)$; this is the natural induced action of the group $\PGL(2, q)\times \Gal(\fqn/\F_q)\le \PTL(2, q^n)$ over irreducible polynomials. For this special subgroup, we obtain finner results on the structure of invariant polynomials. More specifically, we show that the $\mathfrak g$-invariants over $\F_{q^n}$ are irreducible factors (over $\F_{q^n}$) of $[A]$-invariants (over $\F_q$). For more details, see Theorem~\ref{thm:equivalence}. As an application, we obtain the exact number of {\em self-conjugate reciprocal polynomials} of a fixed degree, which is a way more explicit result when compared with the enumeration formula of~\cite{BJU}.

The paper is organized as follows. Section 2 provides background material and preliminary results. In particular, a characterization of the invariant polynomials is presented. In Section 3 we use this characterization and, by emplying similar ideas to the ones in~\cite{ST12}, we prove Theorem~\ref{thm:main}. Finally, in Section 4, we explore the polynomials that are invariant by elements of the group $\PGL(2, q)\times \Gal(\fqn/\F_q)\le \PTL(2, q^n)$.

\section{Preliminaries}
This section provides background material that is used throughout the paper and some preliminary results. For a prime power $q$, let $\I(q)$ be the set of monic irreducible polynomials of degree at least two over the finite field $\F_q$. According to Lemma $2.2$ of~\cite{ST12}, we have the following result.
\begin{lemma} \label{propriedades}
Let $A, B \in \GL(2,q)$, $f , g\in \mathcal{I}(q)$, and let  $I$ be denote the $2\times2$ identity matrix. The following hold: 
\begin{enumerate}[(i)]
\item $A \circ f \in \mathcal{I}(q)$ and $\deg (A \circ f) = \deg (f)$;
\item $ I \circ f = f$;
\item $(AB) \circ f = A \circ (B \circ f)$;
\item $A \circ (f \cdot g) = ( A \circ f) \cdot (A \circ g)$.
\end{enumerate}
\end{lemma}

The previous lemma entails that the group $\PGL(2, q)$ acts on the set $\I(q)$ and the compositions $[A]\circ f$ are degree-preserving. This result is the starting point for extending this action to the group $\PTL(2, q^n)$.

\subsection{The action of $\PTL(2, q^n)$ on irreducible polynomials}

The {\em Projective Semilinear Group} $\PTL(2,q^n)$ is defined as the semi direct product of $\PGL(2,q^n)$ and $\Gal(\F_{q^n}/ \F_q)$ (this last group is isomorphic to the cyclic group of order $n$). The group $\PTL(2, q^n)$ is endowed with the following product: for $[A,\sigma_i], [B, \sigma_j] \in \PTL(2,q^n)$, we set 
$$[A,\sigma_i] \diamond [B,\sigma_j] = [A\sigma_i(B), \sigma_{i+j}],$$
where the index $i$ at $\sigma_i$ is taken modulo $n$. The identity element is $[I,\sigma_n]$, where $I$ is the $2 \times 2$ identity matrix. From the invariant structure of Frobenius automorphisms over the polynomial ring $\fqn[x]$ (like irreducibility and degree), Lemma~\ref{propriedades} readily implies the following result. 
\begin{lemma} \label{lemapropriedades} 
Let $A,B \in \mathrm{PGL}(2,q^n)$, $f, g\in \mathcal{I}_{k}$ with $k\ge 2$ and let $I$ be the $ 2 \times 2$ identity matrix. Then the following hold: 
\begin{enumerate}[(i)]
\item  $[ A, \sigma_i ] \ast  f \in \mathcal{I}_k$, i.e., $[ A, \sigma_i ] \ast  f$ is irreducible and $\deg ([ A, \sigma_i ] \ast  f )=\deg (f)$;
\item  $ [ I,\sigma_n ] \ast  f = f$;
\item  $([B, \sigma_j]\diamond [A, \sigma_i])\ast f=[ B, \sigma_j ]\ast( [A,\sigma_i ] \ast  f)$;
\item  $[ A, \sigma_i ] \ast (f \cdot g) = ( [ A, \sigma_i ] \ast  f)( [ A, \sigma_i ] \ast g) $;
\end{enumerate}
\end{lemma}

In particular, the compositions $[A, \sigma_i]\ast f$ define an action of the group $\PTL(2, q^n)$ on the sets $\I_k$ for $k\ge 2$. From now and on, we are mostly interested in the study of the fixed points of this group action. As follows, we show that it suffices to consider only the elements of $\PTL(2, q^n)$ of a special form.

\subsection{A first reduction}
For a generic element $\mathfrak g=[A, \sigma_i]\in \PTL(2, q^n)$, let us compute its order $\ord(\mathfrak g)$. Set $t=\gcd(i, n)$, hence $\ord(\mathfrak g)$ is divisible by $n/t$. 
If we set $C=A\sigma_i(A)\ldots \sigma_{i(n/t-1)}(A)$, then it is direct to verify that $\ord(\mathfrak g)=n/t\cdot \ord([C])$. In particular, for any $j\ge 1$ relatively prime with $\ord(\mathfrak g)$, we have that $f\in \I_k, k\ge 2$ is $[A, \sigma_i]$-invariant if and only if it is $[A, \sigma_i]^j$-invariant. Since $i/t$ and $n/t$ are relatively prime, there exists a positive integer $a$ such that $a\cdot i/t\equiv 1\pmod {n/t}$ or, equivalently, $ai\equiv t\pmod n$. In addition, we have that $a$ and $n/t$ are relatively prime and so there exist infinitely many prime numbers of the form $n/t\cdot R+a$ by the Dirichlet's Theorem. Pick a prime $P$ of such form in a way that $P>\ord(\mathfrak g)$. Therefore, a polynomial $f\in \I_k$ with $k\ge 2$ is $[A, \sigma_i]$-invariant if and only if it is $[A, \sigma_i]^P$-invariant. However, since $P\equiv a\pmod {n/t}$, we have that $[A, \sigma_i]^P=[B, \sigma_t]$ for a suitable $B\in \GL(2,\fqn)$. In particular, we proved the following result.

\begin{lemma}
For any $\mathfrak g\in \PTL(2, \fqn)$, there exists a divisor $t$ of $n$ and an element $\mathfrak h\in \PTL(2, \fqn)$ of the form $[C, \sigma_t]$ such that the $\mathfrak g$-invariants are exactly the $\mathfrak h$-invariants.
\end{lemma}

If we set $Q=q^t$ and $\tau_{1}=\sigma_{t}$, then $q^n=Q^{n/t}$ and $[A, \sigma_t]=[A, \tau_1]$. In this case, $\tau_1$ is the generator of $\Gal(\F_{Q^{n_0}}/\F_Q)$, where $n_0=n/t$. In other words, there is no loss of generality in considering only elements of the form $[A, \sigma_1]$.

\subsection{On the characterization of invariant polynomials}
Our aim is to provide a characterization of the $[A, \sigma_1]$-invariants for generic $[A]\in \PGL(2, q^n)$. In order to simplify the powers of the elements in $\PTL(2, q^n)$, we introduce the following definition.

\begin{define}
For $A\in \GL(2, q^n)$ and a positive integer $i$, set $$A_i^*=A\sigma_1(A)\ldots \sigma_{i-1}(A).$$ For $i=n$, we simply write $A^*=A_n^*$.
\end{define}
Following the procedure of~\cite{ST12}, we introduce two useful definitions.
\begin{define}
For $A\in \GL(2,q^n)$ with  $A=\left(\begin{matrix}
a&b\\
c&d
\end{matrix}\right)$ and a nonnegative integer $m$, we set
$$F_{A, m}(x)=bx^{q^m+1}-ax^{q^m}+dx-c.$$
For nonnegative integers $i, m$ with $m\ge i$, we set
$$F_{A, m, i}(x)=\sigma_{-i}(F_{A_i^*, m-i}(x)).$$
\end{define}

\begin{define}
For an integer $i$,  $A \in \GL(2, q^n)$ with $A=\left(\begin{matrix}
a&b\\
c&d
\end{matrix}\right)$ and $\alpha\in \overline{\F}_q\setminus \fqn$, we set
$$[A]\circ \alpha=\frac{d\alpha-c}{-b\alpha+a}\quad\text{and}\quad [A, \sigma_i]\ast \alpha= [A]\circ \sigma_i(\alpha)=[A]\circ \alpha^{q^i}.$$
\end{define}

It is direct to verify that the group $\PGL(2, q^n)\rtimes \Z$ acts on the set $\alpha\in \overline{\F}_q\setminus \fqn$ via the compositions $[A, \sigma_i]\ast \alpha$. The following proposition shows that the invariants are exactly the irreducible divisors (over $\fqn$) of a special class of polynomials.

\begin{prop}\label{prop:equiv}
Let $k\ge 2$, $f\in \I_k$ and let $\alpha\in \overline{\F}_q\setminus \fqn$ be any root of $f$. For $[A, \sigma_1]\in \PTL(2, q^n)$, the following are equivalent:

\begin{enumerate}[(i)]
    \item $f$ is $[A, \sigma_1]$-invariant;
    \item $[A, \sigma_1]\ast \alpha=\alpha^{q^{nr}}$ for some nonnegative integer $r$;
    \item $f$ divides $F_{A, nr,1}$ for some nonnegative integer $r$.
\end{enumerate}
\end{prop}

\begin{proof}
We split the proof into cases.
\begin{enumerate}[(i)] 
\item To prove that (i) implies (ii), first notice that
 $$f(\alpha) = 0 \Leftrightarrow ([A,\sigma_1] \ast f)([A, \sigma_1] \ast \alpha) = 0.$$
By assumption, $[A,\sigma_1] \ast f = f$ and so $f([A,\sigma_1] \ast \alpha) =0$. Since $f$ is irreducible, the latter implies that 
 $[A,\sigma_1] \ast \alpha $ is conjugate to $\alpha$ over $\fqn$, and so $[A,\sigma_1] \ast \alpha = \alpha^{q^{nr}}$ for some nonnegative integer $r$.
\item The fact that (ii) implies (iii) follows directly from the definition of the polynomial $F_{A, nr, 1}(x)$.

\item To prove that (iii) implies (i), we observe that the equality $F_{A,nr,1}(\alpha)=0$ is equivalent to $[A, \sigma_1] \ast \alpha = \alpha^{q^{nr}}$. Since $\alpha^{q^{nr}}$ is conjugate to $\alpha$ over $\fqn$, such an element is also a root of $f$. From the proof of item (i), we see that the latter implies the equality $[A, \sigma_1] \ast f = f$.
\end{enumerate}
\end{proof}

\section{On the divisors of the polynomials $F_{A, nr, 1}(x)$}
Proposition~\ref{prop:equiv} entails that, for a fixed $[A]\in \PGL(2, q^n)$, the $[A, \sigma_1]$-invariants are exactly the irreducible factors of the polynomials $F_{A, nr, 1}(x)$ with $r\ge 0$. In particular, the study of the degree distribution of the irreducible factors of $F_{A, nr, 1}$ over $\fqn$ is crucial for the proof of Theorem~\ref{thm:main}. First, we prove that we do not need to search for all $r$ in order to verify if a polynomial is $[A, \sigma_1]$-invariant.

\begin{lemma}\label{lem:search}
Let $f$ be a $m$-degree monic irreducible polynomial over $\F_{q^n}$ and $[A]\in \PGL(2, q^n)$, where $m\ge 1$. If $f$ divides a polynomial of the form $F_{A, nr, 1}$, then the minimal positive $r$ with this property satisfies $r\le m$ and any other positive integer $r'$ with such a property satisfies $r\equiv r'\pmod m$.
\end{lemma}

\begin{proof}
Let $\alpha$ be any root of $f$. Since $f$ is irreducible, $f$ divides $F_{A, nr, 1}$ if and only if $F_{A, nr, 1}(\alpha)=0$. However, from definiton, we see that $F_{A, nr, 1}(\alpha)=F_{A, nr', 1}(\alpha)$ if and only if $\alpha^{q^{n(r-r')}}=\alpha$. Since $f$ is irreducible and has degree $m$, the latter is equivalent to $r\equiv r'\pmod m$. In particular, if $f$ divides $F_{A, nr_0, 1}$ for some $r_0\ge 0$, we can find a positive integer $r\le m$ with $r\equiv r_0\pmod m$. 
\end{proof}

Throughout this section, we fix $[A, \sigma_1]\in \PTL(2, q^n)$ and $D=\ord([A^*])$. We have the following lemma.

\begin{lemma}
For non-negative integers $i, m$ with $m\ge i$, the polynomial $F_{A, m , i}$ is separable.
\end{lemma}
\begin{proof}
Let $A_i^* = \begin{pmatrix}
a_i & b_i \\
c_i & d_i 
\end{pmatrix}$ and suppose that $F_{A, m , i}$ and its formal derivative $F'_{A, m , i}$ have a common zero $\beta \in \overline{\F}_{q^n}$. Therefore,
$$F'_{A, m , i}(\beta) = b_i^{q^{n-i}} \beta^{q^m} + d_i^{q^{n-i}} = (b_i^{q^{n-(m+i)}} \beta + d_i^{q^{n-(m+i)}})^{q^{m}} = 0, $$
and 
\begin{align*}F_{A, m , i}(\beta) & = \beta(b_i^{q^{n-i}}\beta^{q^m} - d_i^{q^{n-i}}) - (a_i^{q^{n-i}} \beta^{q^m} + c_i^{q^{n-i}}) \\
& = - (a_i^{q^{n-(m+i)}} \beta + c_i^{q^{n-(m+i)}})^{q^m} = 0  
\end{align*}
That is $a_i^{q^{n-(m+i)}} \beta + c_i^{q^{n-(m+i)}} = 0$ which gives a non-trivial linear combination $\beta(a_i^{q^{n-i}}, b_i^{q^{n-i}}) + (c_i^{q^{n-i}}, d_i^{q^{n-i}}) = (0,0)$ of the rows of the matrix $\sigma_{n-i}(A_i^*)$, but this is a contradiction, since $A_i^*$ is nonsingular (recall that $A$ is nonsigular).
\end{proof}

As follows, we provide information on the degrees of the irreducible factors of the polynomials $F_{A, nr, 1}$.

\begin{prop}\label{prop:degrees}
Fix $r\ge 1$ an integer. If $f\in \I_k$ with $k>2$ divides $F_{A, nr, 1}$, then $f$ has degree of the form $D\cdot s$, where $D=\ord([A^*])$ and $s$ is a divisor of $nr-1$ with $\gcd\left(\frac{nr-1}{s}, D\right)=1$. In particular, $\gcd(s, n)=1$.
\end{prop}
\begin{proof}
From Proposition~\ref{prop:equiv}, $f$ is $[A, \sigma_1]$-invariant. In particular,
$$[A^*]\ast f=[A^*, \sigma_n]\ast f=[A, \sigma_1]^n\ast f=f,$$
hence $f$ is $[A^*]$-invariant. From Theorem~3.3 of~\cite{ST12}, $f$ has degree divisible by $D=\ord([A^*])$. Write $\deg(f)=Ds$ and let $\alpha\in \overline{\F}_q\setminus \fqn$ be any root of $f$. From Proposition~\ref{prop:equiv}, we have that
$$[A, \sigma_1]\ast \alpha=\alpha^{q^{nr}},$$
hence 
$$\alpha^{q^{Dn}}=[I, \sigma_{Dn}]\ast \alpha=[A, \sigma_1]^{Dn}\ast \alpha=\alpha^{q^{n^2rD}}.$$
Therefore, $\alpha=\alpha^{q^{nD(nr-1)}}$ and so $s$ divides $nr-1$ (recall that $f$ has degree $Ds$ and is irreducible over $\fqn$). We shall prove that $(nr-1)/s$ and $D$ are relatively prime. On the contrary, set $e=\gcd((nr-1)/s, D)>1$ and $l=D/e<D$. Therefore, 
$$[C, \sigma_{nl}]\ast \alpha=[A, \sigma_1]^{nl}\ast \alpha=\alpha^{q^{n^2rl}},$$
where $C=(A^*)^l$. The last equality is equivalent to
$$[C]\circ \alpha=\alpha^{q^{nl(nr-1)}}=\alpha^{q^{nDs\cdot\frac{nr-1}{es}}}=\alpha.$$
However, since $l<D$, we have that $[C]$ is not the identity matrix and so the last equality entails that $\alpha$ satisfies a nontrivial polynomial equation of degree at most two with coefficients in $\fqn$. This is impossible since $\alpha$ is a root of $f$ and $f$ is an irreducible polynomial of degree $k>2$.
\end{proof}

Next, we find bounds on the number of irreducible factors of the polynomials $F_{A, nr, 1}$, according to their degree.

\begin{theorem}\label{thm:aux}
Let $r$ be a positive integer and let $s$ be a divisor of $nr-1$ such that $nr-1=sm$ and $\gcd(m, D)=1$. If $j$ is the least positive integer such that $jm\equiv 1\pmod{Dn}$, then the following hold.

\begin{enumerate}
    \item For each divisor $l>2$ of $Ds$, the $l$-degree irreducible divisors of the polynomial $F_{A, nr, 1}$ are the $l$-degree irreducible divisors of the polynomial $F_{A_j^*, j+s, j}$. In particular, the number of $Ds$-degree irreducible divisors of the polynomial $F_{A, nr, 1}$ is at most $\frac{q^{s}+1}{Ds}$. 
    
    \item The irreducible factors of $F_{A_j^*, j+s, j}$ are of degree $Dt$, where $t$ divides $s$ and of degree at most two.
\end{enumerate}
In particular, the number $N$ of $Ds$-degree irreducible divisors of the polynomial $F_{A, nr, 1}$ satisfies
$$N=\frac{q^s}{Ds}+H(s),$$
where $|H(s)|\le 3q^{e(s)}$ and $e(s)=\max\{2n, s/2\}$.
\end{theorem}
\begin{proof}
We split the proof into cases.

\begin{enumerate}
    \item Write $jm=Dnk+1$, let $l$ be a divisor of $Ds$ and let $f$ be an $l$-degree irreducible factor of $F_{A, nr, 1}$. If $\alpha\in \overline{\F}_q\setminus {\fqn}$ is a any root of $f$, from definition,
    $[A, \sigma_1]\ast \alpha=\alpha^{q^{nr}}$.Therefore, $$[A_j^*, \sigma_j]\ast \alpha=\alpha^{q^{nrj}}=\alpha^{q^{msj+j}}=\alpha^{q^{s+j+Dnsk}}=\alpha^{q^{s+j}}.$$
    The last equality is due to fact that, since $f$ has degree $l$ and $l$ divides $Ds$, we have that
    $\alpha^{q^{Dnsk}}=\alpha$.
    In particular, $\alpha$ is a root of $F_{A_j^*, j+s, j}$, hence such a polynomial is divisible by $f$. Conversely, if $f$ has degree $l$ and divides $F_{A_j^*, j+s, j}$, we have that any root $\alpha$ of $f$ satisfies $[A, \sigma]^j\ast \alpha=\alpha^{q^{s+j}}=\alpha^{q^{nrj}}$ (again, we used the fact that $l$ divides $Ds$). The last equality implies that
    $$[A, \sigma_{mj}]\ast \alpha=\alpha^{q^{nrjm}},$$
    where we used the fact that $[A, \sigma_1]^{mj-1}=[I, \sigma_{mj-1}]$.
    If we set $\beta=\alpha^{q^{Dnk}}=\alpha^{q^{mj-1}}$, the previous equality is equivalent to 
    $$[A, \sigma_1]\ast \beta=\beta^{q^{(nr-1)jm+1}}=\beta^{q^{mk(nDs)+nr}}=\beta^{q^{nr}},$$
    since $\beta$ is a conjugate of $\alpha$ over $\fqn$ and $l$ divides $Ds$. The last equality entails that $\beta$ is a root of $F_{A, nr, 1}$. Since $\alpha$ and $\beta$ are conjugates over $\fqn$, it follows that $f$ divides $F_{A, nr, 1}$. 
    
    We observe that, from definition, the polynomial $F_{A_j^*, j+s, j}$ has degree at most $q^{s}+1$, hence the number of $Ds$-degree irreducible divisors of the polynomial $F_{A, nr, 1}$ is at most $\frac{q^{s}+1}{Ds}$. 
    
\item Let $f$ be an $l$-degree irreducible factor of $F_{A_j^*, j+s, j}$, where $l>2$, and pick $\alpha\in \overline{\F}_q\setminus\fqn$ a root of $f$. From definition, $[A_j^*, \sigma_j]\ast \alpha=\alpha^{q^{s+j}}$, hence \begin{equation}\label{eq:key}[A_j^*, \sigma_{n-s}]\ast \beta=\beta^{q^n},\end{equation} 
    where $\alpha=\beta^{q^{n-s-j}}$. Let $g$ be the minimal polynomial of $\beta$ over $\fqn$. From definition, $g=\sigma_{n-s-j}(f)$ and so $g$ also has degree $l$. Additionally, from definition, $j\equiv n-s\pmod n$ and so Eq.~\eqref{eq:key} entails that $g$ is $[A_j^*, \sigma_j]$-invariant. Since $j$ is relatively prime with $nD$ (the order of $[A, \sigma_1]$), it follows that $g$ is also $[A, \sigma_1]$-invariant. Since $l>2$, it follows from Propositions~\ref{prop:equiv}
 and~\ref{prop:degrees} that $l=Dt$ for some positive integer $t$. Again, since $j\equiv n-s\pmod n$, we have that $\sigma_{n-s}(A_j^*)=\sigma_j(A_j^*)$ and so Eq.~\eqref{eq:key} entails that
 $$[A_{jnD}^*, \sigma_{nD(n-s)}]\ast \beta=[A_j^*, \sigma_j]^{nD}\ast \beta=\beta^{q^{n^2D}}.$$
    However, $[A_{jnD}^*]=[(A^*)^{jD}]=[I]$ and so the previous equality implies that
    $\beta^{q^{nD(n-s)}}=\beta^{q^{n^2D}}$ or, equalivalently, $\beta^{q^{nDs}}=\beta$. In particular, $l$ divides $Ds$, i.e., $t$ divides $s$. Therefore, $f$ has degree of the form $Dt$, where $t$ divides $s$.
\end{enumerate}
We observe that, from definition, $F_{A_j^*, j+s, j}$ is separable and has degree $q^s$ or $q^{s}+1$. For each divisor $t$ of $s$, let $C_t$ be the set of the $Dt$-degree irreducible divisors of $F_{A_j^*, j+s, j}$ and set $N=|C_s|$. If $\mu$ denotes the number of irreducible divisors of $F_{A_j^*, j+s, j}$ of degree at most two, item (ii) entails that
\begin{equation}
q^{s}+1\ge N\cdot Ds+\sum_{t|s\atop{t<s}}|C_t|\cdot Dt\ge q^{s}-2\mu    
\end{equation}

We observe that the upper bound is due to fact that the equalities $Dt=1$ or $Dt=2$ may occur. Fix $t$ a divisor of $s$ and write $nr-1=t(s/tm)$. Item 1 and Proposition~\ref{prop:degrees} entail that either $C_t$ is empty or $s/t$ is relatively prime with $D$ and, in the former case, we have that $C_t$ has at most $\frac{q^t+1}{Dt}$ elements. In any case, $Dt\cdot |C_t|\le q^t+1$. Therefore, 
$$\sum_{t|s\atop{t<s}}|C_t|\cdot Dt\le \sum_{1\le i\le s/2}(q^t+1)\le s(q^{s/2}+1)/2.$$
In particular, 
$$\frac{q^s}{Ds}+1>N>\frac{q^s-2\mu -s(q^{s/2}+1)/2}{Ds}>\frac{q^s}{Ds}-3q^{e(s)},$$
where $e(s)=\max\{2n, s/2\}$, since the number of monic irreducible polynomials of degree at most two over $\fqn$ equals $q^{2n}$ and $1<q^{s/2}$.
\end{proof}

It remains to prove that, for any $s$ such that $\gcd(s, n)=1$, there exist positive integers $r$ for which $s$ divides $nr-1$ and $nr-1=sm$ with $\gcd(m, D)=1$. In the following lemma, we obtain the exact number of incongruent values of $r$ modulo $Ds$ with such a property.

\begin{lemma}\label{lem:count}
Let $s, n, D$ be positive integer such that $\gcd(s, n)=1$. Then there exists $\varphi(D)$ positive integers $r\le Ds$ such that $s$ divides $rn-1$ with $rn-1=sm$ and $\gcd(m, D)=1$.
\end{lemma}
\begin{proof}
Let $R$ be the least positive integer such that $Rn\equiv 1\pmod n$ and set $M=\frac{Rn-1}{s}$. Hence, $s$ divide the integers $r_in-1$, where $r_i=R+i\cdot s$, with $0\le i< D$. Additionally, $r_in-1=M+i$. We observe that the integers $r_i$ are all between $1$ and $Ds$. We observe that the set $\{M+i\}$ with $0\le i <D$ is a set of complete residues modulo $M$, hence the number of integers $r_i$ such that $\gcd (M+i, D)=1$ is, by definition, $\varphi(D)$.
\end{proof}

It is straightforward to see that the previous lemma, combined with Theorem~\ref{thm:aux} and Lemma~\ref{lem:search}, provide the proof of Theorem~\ref{thm:main}.

\section{The special subgroup $\PGL(2, q)\times \Gal(\fqn/\F_q)$}
In this section, we study the $[A, \sigma_i]$-invariants in the special case that $[A]\in \PGL(2, q)$. In other words, we are restricting our action to the group $$\PGL(2, q)\times \Gal(\fqn/\F_q)\le \PTL(2, q^n).$$ As pointed out earlier, there is no loss of generality on considering elements of the form $[A, \sigma_1]$. In particular, we may only look at elements of the form $[A, \sigma_i]$ with $[A]\in \PGL(2, q)$ and $1\le i\le n$ such that $\gcd(i, n)=1$.

We fix $[A]\in \PGL(2, q)$,  $d=\ord([A])$ and $d_0=\gcd(d, n)$. Therefore, the order of $[A, \sigma_i]$ equals $nd/d_0$. In the notation of Theorem~\ref{thm:main}, we have that $[A^*]=[A]^n$ and so $D=\ord([A^*])=d/d_0$. Following the same steps of Propositions~\ref{prop:equiv} and~\ref{prop:degrees}, we can show that the $[A, \sigma_i]$-invariants of degree greater than $2$ have degree of the form $\frac{d}{d_0}\cdot s$, where $\gcd(s, n)=1$. As follows, we have a nice relation between $[A, \sigma_j]$-invariants with $\gcd(j, n)=1$ and the $[A]$-invariants considered in~\cite{ST12}.

\begin{theorem}\label{thm:equivalence}
Fix $[A]\in \PGL(2, q)$ an element of order $d$ and set $d_0=\gcd(d, n)$. Then, for any $s$ such that $\gcd(s, n)=1$ and any monic irreducible polynomial $f\in \F_{q^n}[x]$ of degree $\frac{d}{d_0}\cdot s>2$, the following are equivalent:

\begin{enumerate}[(i)]
\item $f$ is $[A, \sigma_i]$-invariant for some $1\le i\le n$ such that $\gcd(i, n)=1$;
\item  $f$ is $[A, \sigma_i]$-invariant for some $1\le i\le n$ such that $\gcd(i, n)=1$ and the least positive integer $t$ such that $f\in \F_{q^t}[x]$ is $t=d_0$;
\item $f$ divides a $ds$-degree monic irreducible polynomial $G\in \F_q[x]$ that is $[A]$-invariant, i.e., $[A]\circ G=G$.
\end{enumerate}
\end{theorem}

\begin{proof}
We split the proof into cases.
\begin{itemize}
\item To see that (i) implies (ii), suppose that $f$ is $[A, \sigma_i]$-invariant for some $1\le i\le n$ such that $\gcd(i, n)=1$ and let $t$ be the least positive integer such that $f\in \F_{q^t}[x]$. Since $f\in \F_{q^n}[x]$, we have that $t$ divides $n$. In particular, we have that $f$ is also $[A, \sigma_i]^t=[A^{t}, \sigma_{ti}]$-invariant and, since $f\in \F_{q^t}[x]$, it follows that $[A]^t\circ f=f$, i.e., $f$ is $[A]^t$-invariant. From Theorem 4.5 in \cite{ST12}, $\frac{d}{d_0}\cdot s>2$ is divisible by $\ord([A]^t)=d/\gcd(t, d)$. Since $t$ divides $n$ and $s$ is relatively prime with $n$, we necessarily have that $\gcd(t, d)=d_0$ and so $t$ is divisible by $d_0$. However, we also have that $f$ is $[A, \sigma_i]^d=[I, \sigma_{id}]$-invariant and so $f\in \F_{q^{id}}[x]$. Since $\gcd(n, id)=d_0$, we have that $f\in \F_{q^{d_0}}[x]$ and so $t$ divides $d_0$. In conclusion, $t=d_0$.
\item To see that (ii) implies (iii), suppose that  $f$ is $[A, \sigma_i]$-invariant for some $1\le i\le n$ such that $\gcd(i, n)=1$ and the least positive integer $t$ such that $f\in \F_{q^t}[x]$ is $t=d_0$. In particular, $\sigma_{d_0}(f)=f$ and so $f$ is $[A]^{d_0}$-invariant, i.e., $[A]^{d_0}\circ f=f$. Additionally, if $G\in \F_q[x]$ denotes the unique monic irreducible polynomial that is divisible by $f$, we have that $G$ is just the minimal polynomial (over $\F_q$) of any root of $f$. From the minimality of $d_0$, we necessarily have that
\begin{equation}\label{eq:mobius-factor}G(x)=f(x)\cdot \prod_{j=1}^{d_0-1}\sigma_j(f(x))=f(x)\cdot \prod_{j=1}^{d_0-1}\sigma_{ij}(f(x)),\end{equation}
where in the last equality we used the fact that $\sigma_l(f(x))=\sigma_{l_0}(f(x))$ whenever $l\equiv l_0\pmod {d_0}$ and $\gcd(i, d_0)=1$. However, since $f$ is $[A, \sigma_i]$-invariant, we have that $\sigma_i(f(x))=[A]^{-1}\circ f$ and, more generally, $\sigma_{ij}(f(x))=[A]^{-j}\circ f$ for any $j\in \Z$. Therefore, Eq.~\eqref{eq:mobius-factor} and the equality $[A]^{d_0}\circ f=f$ yields
$$G=\prod_{j=1}^{d_0}[A]^{-j}\circ f(x)=\prod_{j=1}^{d_0}[A]^{j}\circ f(x).$$
Again, since $[A]^{d_0}\circ f=f$, we have that $[A]\circ G=G$. 

\item To see that (iii) implies (i), suppose that $f$ divides a $ds$-degree monic irreducible polynomial $G\in \F_q[x]$ that is $[A]$-invariant. Since $f\in \F_{q^n}[x]$ and $\gcd(n, ds)=d_0$, we necessarily have that $G$ splits into $d_0$ monic irreducible polynomials over $\F_{q^{d_0}}$, each of degree $\frac{d}{d_0}\cdot s$. In particular, $f\in \F_{q^{d_0}}[x]$. In addition, since $G\in \F_q[x]$, the factorization of $G$ over $\F_q$ is given as follows
$$G(x)=f(x)\cdot \prod_{j=1}^{d_0-1}\sigma_j(f(x)).$$
Since $G$ is $[A]$-invariant, we have that $[A]\circ f=\sigma_{d_0-j}(f)$ for some $0\le j\le d_0-1$. We claim that $\gcd(j, d_0)=1$. Set $k=\gcd(j, d_0)$. Therefore, $d_0(d_0-j)/k$ is divisible by $d_0$ and so the equality $[A]\circ f=\sigma_{d_0-j}(f)$ yields
$$[A]^{d_0/k}\circ f=\sigma_{\frac{d_0(d_0-j)}{k}}(f)=f.$$
We observe that the element $[A]^{d_0/k}$ has order $dk/d_0$. Therefore, from Theorem 4.5 in \cite{ST12}, $dk/d_0$ divides the degree of $f$, i.e., $k$ divides $s$. However, $k$ divides $d_0$, hence it divides $n$. Since $\gcd(s, n)=1$, we have that $k=1$. In particular, $f$ is $[A, \sigma_i]$-invariant for some $i\le d_0$ such that $\gcd(i, d_0)=1$. Since $f\in \F_{q^{d_0}}[x]$, the latter implies that $f$ is $[A, \sigma_{i+d_0l}]$-invariant for any positive integer $l$. Since $\gcd(i, d_0)=1$, the Dirichlet Theorem ensures the existence of infinitely many prime numbers of the form $i+d_0l$. If we pick a prime $P$ of this form in a way that $P>n$ and $p_0$ is the least positive integer such that $P\equiv p_0\pmod n$, we have that $f$ is $[A, \sigma_{p_0}]$-invariant, $1\le p_0\le n$ and $\gcd(p_0, n)=1$. 
\end{itemize}
\end{proof}

\subsection{On self-conjugate-reciprocal polynomials}
In~\cite{BJU}, the authors introduced the so called {\em self-conjugate-reciprocal irreducible monic} (SCRIM) polynomials. These are the monic irreducible polynomials $f\in \F_{q^2}[x]$ (distinct from $x$) such that its \emph{monic reciprocal} $f^*(x):=f(0)^{-1}\cdot x^{\deg(f)}f(1/x)$ coincides with its \emph{conjugate} $\sigma_1(f)$ over $\F_q$, i.e., 
$$f^*(x)=\sigma_1(f).$$
This is a nice variation of the so called {\em self-reciprocal irreducible monic} (SRIM) polynomials~\cite{MG90}, which are the polynomials satisfying $f^*(x)=f(x)$.
The authors explored the existence and number of these SCRIM polynomials and some of their results can be compiled as follows.

\begin{theorem}\label{them:BJU}
The degree of any SCRIM is odd. In addition, if $n$ is odd and $D_n$ denotes the set of divisors of $q^n+1$ that do not divide any $q^k+1$ with $0\le k\le n-1$, then the number of $n$-degree SCRIM's equals
$$\frac{1}{n}\sum_{d\in D_n}\varphi(d).$$
\end{theorem}

For more details, see~Theorems~3.3 and~3.15 of~\cite{BJU}. More recently, the same authors used such polynomials in the study of Hermitian codes over finite rings~\cite{BJU-2}. We observe that this conjugate-reciprocal operation can be brought to the context of the present paper. In fact, for $B=\left(\begin{matrix}
0&1\\
1&0
\end{matrix}\right)$ and $\mathfrak g_0=[B, \sigma_1]\in \PTL(2, q^2)$, we have that the $\mathfrak g_0$-invariants are exactly the SCRIM's. Notice that $[B]$ has order $2$. In particular, we obtain the following result.

\begin{cor}\label{cor:scrim}
Fix $n$ an odd integer. Then the $n$-degree monic irreducible polynomial $f\in \F_{q^2}[x]$ is an SCRIM polynomial if and only if $f$ divides a $2n$-degree self-reciprocal irreducible monic polynomial $G\in \F_q[x]$. In particular, the number of $n$-degree SCRIM's equals
$$\frac{1}{n}\sum_{d|n}\mu(d)q^{n/d}.$$
\end{cor}
\begin{proof}
The fact that the $n$-degree SCRIM polynomials over $\F_{q^2}$ are exactly the irreducible factors of the $2n$-degree self reciprocal polynomials over $\F_q$ follows directly from Theorem~\ref{thm:equivalence}. Therefore, if $a(n)$ is the number of $n$-degree SCRIM's and $b(n)$ is the number of $2n$-degree self reciprocal irreducible monic polynomials over $\F_q$, we have that $a(n)=2\cdot b(n)$. In addition, according to Theorem 3 of~\cite{MG90}, for $n$ odd,
$$b(n)=\frac{1}{2n}\sum_{d|n}\mu(d)q^{n/d},$$
and the result follows.
\end{proof}

We observe that, in the notation of Theorem~\ref{thm:main}, the number of $s$-degree SCRIM's equals $n_B(s)$ and the same theorem provides $n_B(s)\approx \frac{q^s}{s}$ if $s$ is odd. This asymptotic formula agrees with the exact formula given in Corollary~\ref{cor:scrim}. 

\begin{remark}
We comment that Corollary~\ref{cor:scrim} also suggests a way of constructing SCRIM's of any degree. For this, we just pick a self-reciprocal irreducible monic polynomial $f\in \F_q[x]$ of degree $2n$ and so $f$ factors over $\F_{q^2}$ as $f(x)=g(x)\sigma_1(g(x))$ with $g\in \F_{q^2}[x]$ an $n$-degree irreducible polynomial. Then, from Theorem~\ref{thm:equivalence}, both $g$ and $\sigma_1(g)$ are SCRIM polynomials of degree $n$. The construction of self-reciprocal polynomials over finite fields is discussed in~\cite{MG90}.
\end{remark}

\subsubsection{Additional Remarks}
We observe that Theorem~\ref{thm:equivalence} can be applied in a more general situation. In fact, Corollary~\ref{cor:scrim} follows from the fact that the SCRIM polynomials are the $\mathfrak g$-invariants for an element $\mathfrak g$ of the form $[B, \sigma_1]$, where $[B]$ is an involution. In particular, using Theorem~\ref{thm:main} and following the proof of Corollary~\ref{cor:scrim}, we obtain the following result.

\begin{cor}\label{cor:involution}
Let $[B]\in \PGL(2, q)$ be an element of order $2$. Then the $[B, \sigma_1]$-invariant polynomials of $\F_{q^2}$ are of odd degree or of degree at most two. In addition, for each $n\ge 3$ odd, the $n$-degree $[B, \sigma_1]$-invariants are exactly the irreducible factors (over $\F_{q^2}$) of the $2n$-degree $[B]$-invariants (over $\F_q$). 
\end{cor}

In particular, if $[B]\in\PGL(2, q)$ is an involution and $n>1$ is odd, the number of $n$-degree $[B, \sigma_1]$-invariants over $\F_{q^2}$ is exactly twice the number of $2n$-degree $[B]$-invariants over $\F_q$. Exact enumeration formulas for the number of $[B]$-invariants for a generic involution $[B]\in \PGL(2, q)$ is given in~\cite{MP17}.

\begin{center}{\bf Acknowledgments}\end{center}
The second author was partially supported by FAPEMIG  APQ-02973-17, Brazil. 
The third author was supported by FAPESP 2018/03038-2, Brazil.


\end{document}